\documentclass[a4paper,12pt,oneside]{article}
\usepackage[english]{babel}
\usepackage[T1]{fontenc} 
\usepackage[utf8]{inputenc}
\usepackage{amsthm}
\usepackage{bbm}
\usepackage{amsmath}
\usepackage{amssymb}  
\usepackage{indentfirst}
\usepackage{fancyhdr}
\usepackage{amsthm}
\usepackage{graphicx}
\usepackage{pdfpages}
\usepackage{esint}
\usepackage{url}

\numberwithin{equation}{section}

\theoremstyle{plain} 
\newtheorem{thm}{Theorem}[section] 
\newtheorem{cor}[thm]{Corollary} 
\newtheorem{lem}[thm]{Lemma} 
\newtheorem{prop}[thm]{Proposition} 
\newtheorem{rmk}[thm]{Remark} 
\newtheorem{dfn}[thm]{Definition}

\usepackage{xcolor}                    
\definecolor{custom-blue}{RGB}{0,99,166} 
\usepackage{hyperref}
\hypersetup{colorlinks=true, allcolors=custom-blue}

\usepackage{geometry}
\allowdisplaybreaks[4]

\def\t{\tau}

\begin{document}

\author{Authors}

\author{$\text{\sc{Antonio Giuseppe Grimaldi}}^\clubsuit$ \sc{and} $\text{\sc{Stefania Russo}}^\spadesuit$}

\title{{Weak comparison principle for widely degenerate elliptic equations}}

\maketitle
\maketitle

\begin{abstract}
We prove a comparison principle for local weak solutions to a class of widely degenerate elliptic equations of the form
\begin{equation}
     -\text{div} \left(  \left(|Du|-1 \right)^{p-1}_+\frac{Du}{|Du|} \right) = f(x,u) \qquad \text{ in    } \Omega,\notag
 \end{equation}
  where $p \ge 2$ and $\Omega$ is an open subset of $\mathbb{R}^{n}$, $n\geq2$.
 Moreover, we establish some second order regularity results of the solutions, {that yields a weighted Sobolev inequality with widely degenerate weights.} 
\end{abstract}

\medskip
\noindent \textbf{Keywords:} {{Weak comparison principle; degenerate elliptic equations; second order regularity.}}
\medskip \\
\noindent \textbf{MSC 2020:} {{35B51; 35J70.}}

\let\thefootnote\relax\footnotetext{
			\small $^{\clubsuit}$Dipartimento di Ingegneria, Università degli Studi di Napoli ``Parthenope'',
Centro Direzionale Isola C4, 80143 Napoli, Italy. E-mail: \textit{antoniogiuseppe.grimaldi@collaboratore.uniparthenope.it}}

\let\thefootnote\relax\footnotetext{
			\small $^{\spadesuit}$Dipartimento di Matematica e Applicazioni ``R. Caccioppoli'', Università degli Studi di Napoli ``Federico II'', Via Cintia, 80126 Napoli,
 Italy. E-mail: \textit{stefania.russo3@unina.it}}

 \section{Introduction} 
 {In this paper, we consider local weak solutions to widely degenerate elliptic equations of the form
 \begin{equation}\label{Functional}
     -\text{div} \left(  \left(|Du|-1 \right)^{p-1}_+\frac{Du}{|Du|} \right) = f(x,u) \qquad \text{ in    } \Omega,
 \end{equation}
 where $p \ge 2$, $\Omega$ is an open subset of $\mathbb{R}^{n}$, $n\geq2$, $(\,\cdot\,)_{+}$
stands for the positive part. We assume that, for every open set $\Omega' \Subset \Omega$ and every $K>0$, there exist $L,M$ positive constants such that the right-hand side $f: \Omega \times \mathbb{R} \to \mathbb{R}$ satisfies
\begin{equation}\label{F1}
  |f(x,s)-f(x,t)|\le L |s-t|  \tag{F1}
\end{equation}
\begin{equation}\label{F2}
  |f(x,s)-f(y,s)|\le M |x-y|  \tag{F2}
\end{equation}
for all $x,y \in \Omega'$ and every $|s|,|t| \le K$. {As an example of right-hand side, we can consider $f(x,u)=h(x)g(u)$, with $h,g$ both Lipschitz continuous.}

Here, we prove some regularity results for Lipschitz continuous solutions $u$ to \eqref{Functional}, such as second order regularity results of $Du$ and summability properties of the inverse of the weight ${(|Du|-1)_+^{p-1}}$. Exploiting these results, then we are able to establish a weak comparison principle for elliptic equations as in \eqref{Functional}, under suitable assumptions on the right-hand side $f$.

Ciraolo \cite{Ciraolo} proved a weak comparison principle for weak solutions to \eqref{Functional} assuming a priori that $H_{p-1}(Du) \in W^{1,1}(\Omega)$. This global higher differentiability result yields a weak comparison principle in $\Omega$ {and then the uniqueness of the solution}. Since in this paper we deal with local regularity results, we obtain a comparison principle on every open set $\Omega' \Subset \Omega$. In order to extend the comparison principle to the whole $\Omega$, we need to assume that the degeneracy set does not touch the boundary of $\Omega$.

The peculiarity of equation \eqref{Functional}
is that it is uniformly elliptic only outside the ball centered at
the origin with radius $1$, where its principal part behaves
asymptotically as the classical $p$-Laplace operator. Therefore,
the study of such an equation fits into the wider class of the \textit{asymptotically
regular problems}, that were first studied by Chipot and Evans in \cite{CE}. 

One of the main motivations to the study of widely degenerate equations comes from the applications. Indeed, the equation at \eqref{Functional} naturally arises as a model for optimal transport problems with congestion effects.
 We refer to \cite{Bra1,Bra,CJS} and reference therein for a
detailed derivation of the model.

The regularity of weak solutions to \eqref{Functional}, with $f=f(x)$, has been an active field of research in recent years. In \cite{brasco1,Bra,clop}
Lipschitz regularity of solutions has been established under suitable
assumptions on the right-hand side $f$.
In general, no more than Lipschitz regularity can be expected for solutions of equations as in \eqref{Functional}. Indeed, when $f=0$, every $1$-Lipschitz continuous function solves \eqref{Functional}. On the other hand, in \cite{BoDuGiPa,Gri} it was shown that solutions to \eqref{Functional} are of class $\mathcal{C}^1$ outside the degeneracy set.
The higher differentiability of the gradient of solutions is obtained for a datum ${f=f(x)}$ in \cite{Am,AGPdN,Bra,clop, Russo}.

{Much more recently}, in \cite{Ciraolo}, it has been proved a weak comparison principle for Lipschitz continuous solutions to \eqref{Functional}, assuming that the right-hand side $f=f(x)$ is such that the measure of the set where $f$ vanishes is zero. This guarantees that the measure of the degeneracy set is zero.

In \cite{Damascelli} the authors proved a $W^{2,2}$-regularity for solutions to the less degenerate $p$-Laplace equation
$$- \text{div} (|Du|^{p-2}Du)=f(u),$$
{that}, adding a sign assumption on $f$, yields a comparison principle{on small domains} for the solutions, {which does not guarantee the uniqueness of the solution}. For
a list of results on this subject,
we refer to \cite{Riey,Esposito,Montoro,Riey1,Sciunzi,Sciunzi1} and references therein.

{Chipot \cite{ch} studied operators with a $p$-Laplacian structure of the type
\begin{equation}
    - \text{div} \left( \sum_{i=1}^n a_i(x,u)|Du|^{p_i-2}Du \right)=f(x). \label{eqch}
\end{equation}
Assuming that the coefficients $a_i(x,u)$ are continuous and Lipschitz continuous in the $u$-variable, then the solution to \eqref{eqch} is unique.}

As far as we know, there are no regularity results for equations as in \eqref{Functional}, when the datum $f$ depends not only on $x$, but also on $u$. 
Here, the main novelty is to consider widely degenerate equations in which the right-hand side allows an explicit dependence on the solution. 

{It is worth noticing that the equation at \eqref{Functional} is not only degenerate, but in some aspects it behaves as an equation with non-standard growth. Actually, introducing the ellipticity ratio as 
$$\mathcal{R}(\xi):= \dfrac{ \text{the highest eigenvalue of } DH_{p-1}(\xi)}{\text{the lowest eigenvalue of } D H_{p-1}(\xi)}$$
{(for the definition of the function $H_{p-1}$ see \eqref{eq:Hfun}), we have,
by virtue of Lemma \ref{lemH}, that $$\mathcal{R}(\xi) \approx \dfrac{|\xi|}{(|\xi|-1)_+} ,$$
which is clearly unbounded for $|\xi| \to 1$. Hence,
the arguments used in \cite{Damascelli} and the subsequent papers do not seem to work in our case.}

The proof of the weak comparison principle is divided into several steps, consisting in proving that Lipschitz continuous solutions $u$ to \eqref{Functional} satisfy suitable regularity properties, that are of some interest in themselves. \\First, we establish a local higher differentiability result for the function $H_{p-1}(Du)$ (see Section \ref{Pre} for the definition), that allows to differentiate the equation \eqref{Functional} and then to prove that $D^2u$ belongs to a weighted Sobolev space. Next, assuming that $f$ is a positive function, we obtain some integrability results for the inverse of $(|Du|-1)_+$, {and this allows us to prove that the measure of the degeneracy set is zero, i.e.}
$$|\{ x \in \Omega' : |Du| \le 1 \}|=0 \qquad \text{ for any } \Omega' \Subset \Omega.$$
Eventually, using this property of the degeneracy set, we show the validity of the weak comparison principle.}

 \section{Notation and Preliminary Results}\label{Pre}
 In what follows, $C$ or $c$ will denote a general constant that may vary on different occasions (even within the same line of estimates). The dependencies on parameters and special constants will be emphasized using either parentheses or subscripts. The norm we use on $\mathbb{R}^k$, $k \in \mathbb{N}$, will be the standard Euclidean ones and denoted by $|\cdot|$. For $x\in\mathbb{R}^n$ and $r>0$, the symbol $B(x,r)=B_r(x)$ denotes the ball of radius $r$ and center $x$. If the center $x$ is not relevant we will omit it and write only $B_r$.

 For further needs, we
introduce the auxiliary function $H_{\gamma}:\mathbb{R}^{n}\rightarrow\mathbb{R}^{n}$
defined by 
\begin{equation}
H_{\gamma}(\xi):=\begin{cases}
\begin{array}{cc}
(\vert\xi\vert-1)_{+}^{\gamma}\,\frac{\xi}{\left|\xi\right|} & \,\,\mathrm{if}\,\,\,\xi\neq0,\\
0 & \,\,\mathrm{if}\,\,\,\xi=0.
\end{array}\end{cases}\label{eq:Hfun}
\end{equation}

We recall the definition of local weak solution.
\begin{dfn}
    A function $u \in W^{1,p}_{loc}(\Omega) $ is a local weak solution to \eqref{Functional} if, for every $\tilde{\Omega} \Subset \Omega$, it holds 
    \begin{equation}\label{Weaksol}
        \int_{{\Omega}} \langle H_{p-1} (D u), D \varphi \rangle \, dx \,= \, \int_{{\Omega}} f(x,u) \varphi \, dx,
    \end{equation}
    for  any test function $\varphi \in W^{1,p}_0 (\tilde{\Omega})$.
    \\
\end{dfn}

\subsection{Algebraic Inequality}\label{secai}
In this section, we state the relevant algebraic inequality that will be needed
later on.

\begin{lem}\label{alglem}
    Let $p \ge 2$. Then, for any $a ,b> 0$, we have
    $$\dfrac{(a-1)_+^p+(b-1)_+^p}{a^2+b^2} \ge \dfrac{(a-1)_+^p}{4a^2}+\dfrac{(b-1)_+^p}{4b^2}.$$
\end{lem}
\proof If $0 < a \le 1$ or $0 < b \le 1$, then the claim immediately follows.
\\ Now, let $a,b>1$. We may assume that $1 < a \le b$, otherwise we interchange the role of
$a$ and $b$. Since $a^2+b^2 \le 2 b^2$, we have
$$\frac{1}{a^2+b^2} \ge \frac{1}{2 b^2},$$
which implies
\begin{equation}
    \dfrac{(a-1)^p+(b-1)^p}{a^2+b^2} \ge \dfrac{(a-1)^p+(b-1)^p}{2b^2} \ge \dfrac{(b-1)^p}{2b^2} .\label{alineq}
\end{equation}
Consider the function
$$g(t)=\dfrac{(t-1)_+^p}{t^2}, \qquad t > 0.$$
Since $p \ge 2$, then, for $t>1$, we get
$$g'(t)= \dfrac{(t-1)^{p-1}[(p-2)t+2])}{t^3} > 0.$$
Hence, by \eqref{alineq} and the monotonicity of the function $t \in (1,\infty) \mapsto g(t)$, we deduce that
\begin{equation}
    \dfrac{(a-1)^p+(b-1)^p}{a^2+b^2} \ge\dfrac{(b-1)^p}{2b^2} \ge \dfrac{(a-1)^p}{2a^2} .\label{alineq1}
\end{equation}
Combining \eqref{alineq} and \eqref{alineq1}, we get
\begin{equation*}
   2 \, \dfrac{(a-1)^p+(b-1)^p}{a^2+b^2} \ge\dfrac{(b-1)^p}{2b^2} +\dfrac{(a-1)^p}{2a^2} .\label{alineq2}
\end{equation*}
This completes the proof.
\endproof


\subsection{Properties of the Vector Field $H_{p-1}$ }

For the function
$H_{p-1}$ defined by (\ref{eq:Hfun}) with $\gamma=p-1$, we record
the following estimates (see \cite[Lemma 4.1]{Bra}).
\begin{lem}\label{lemmaHalpha2}
    If $ \, 2 \le p < \infty$, then the following estimates 
    \begin{equation}
        \langle H_{p-1}(\xi)- H_{p-1}(\eta), \xi-\eta \rangle \, \geq \, \frac{4}{p^2}\, |H_{\frac{p}{2}}(\xi) - H_{\frac{p}{2}}(\eta)|^2 \label{monprop1}
    \end{equation}
and
      \begin{equation*}
           |H_{p-1}(\xi)- H_{p-1}(\eta)| \leq (p-1) \left( |H_{\frac{p}{2}}(\xi)|^{\frac{p-2}{p}} + |H_{\frac{p}{2}}(\eta)|^{\frac{p-2}{p}} \right) |H_{\frac{p}{2}}(\xi) - H_{\frac{p}{2}}(\eta)|
    \end{equation*}
    hold for every $\xi, \eta \in \mathbb{R}^{n}$. 
\end{lem}

We state a monotonicity property for the function $H_{p-1}$, whose proof can be found in \cite[Lemma 2.8]{BoDuGiPa}.
\begin{lem}
\label{lem.mon} Let {$1 < p<\infty$}. Then, 
for every $\xi,\eta\in\mathbb{R}^{n}$, with $\vert\eta\vert>1$, we have 
\[
\langle H_{p-1}(\xi)-H_{p-1}(\eta),\xi-\eta\rangle\,\geq\,\frac{\min\,\{1,p-1\}}{2^{p+1}}\,\,\frac{(\vert\eta\vert-1)^{p}}{\vert\eta\vert\,(\vert\xi\vert+\vert\eta\vert)}\,\vert\xi-\eta\vert^{2}.
\]

\end{lem}

{Combining Lemmas \ref{lemmaHalpha2} and \ref{lem.mon}, we have the following}
\begin{lem}\label{ellcond}
   Let {$2 \le p<\infty$}. Then, there exists a positive constant $c=c(p)$ such that 
for every $\xi,\eta\in\mathbb{R}^{n} \setminus \{0 \}$, we have 
\begin{equation}\label{monprop}
    \langle H_{p-1}(\xi)-H_{p-1}(\eta),\xi-\eta\rangle\,\geq\,c(p)\,\left[\frac{(\vert\xi\vert-1)_+^{p}}{\vert\xi\vert^2}+\frac{(\vert\eta\vert-1)_+^{p}}{\vert\eta\vert^2} \right]\,\vert\xi-\eta\vert^{2}.
\end{equation}
\end{lem}
\proof 
By virtue of the inequality \eqref{monprop1} in Lemma \ref{lemmaHalpha2}, the left-hand side of \eqref{monprop} is non negative. This, together with Lemma \ref{lem.mon}, yields that there exists a constant $c=c(p)>0$ such that for any $\xi, \eta\in \mathbb{R}^n \setminus \{0\}$, we have
$$\langle H_{p-1}(\xi)-H_{p-1}(\eta),\xi-\eta\rangle\,\geq\,c(p)\,\frac{(\vert\eta\vert-1)_+^{p}}{\vert\eta\vert\,(\vert\xi\vert+\vert\eta\vert)}\,\vert\xi-\eta\vert^{2}$$
and
$$\langle H_{p-1}(\xi)-H_{p-1}(\eta),\xi-\eta\rangle\,\geq\,c(p)\,\frac{(\vert\xi\vert-1)_+^{p}}{\vert\xi\vert\,(\vert\xi\vert+\vert\eta\vert)}\,\vert\xi-\eta\vert^{2}.$$
Using that $|\xi|, |\eta| \le |\xi|+ |\eta|$ and summing the previous inequalities, we get
\begin{equation}
   2 \, \langle H_{p-1}(\xi)-H_{p-1}(\eta),\xi-\eta\rangle\, \ge
\,c(p)\,\frac{(\vert\xi\vert-1)_+^{p}+(\vert\eta\vert-1)_+^{p}}{(\vert\xi\vert+\vert\eta\vert)^2} \,\vert\xi-\eta\vert^{2}. \label{monprop3}
\end{equation}
Finally, applying Lemma \ref{alglem} in the right-hand side of \eqref{monprop3}, we obtain the desired inequality \eqref{monprop}.
\endproof

The following Lemma provides the revelant ellipticity and boundedness properties of the bilinear form $DH_{p-1}(z)$, whose proof can be found in \cite{AGPdN}.

\begin{lem}\label{lemH}
Let $z\in\mathbb{R}^{n}\setminus\{0\}$ and $p \ge 2$.
Then, for every $\zeta\in\mathbb{R}^{n}$ we have 
\begin{align*}
\frac{(|z|-1)_+^{p-1}}{|z|}\,|\zeta|^{2}\le\langle DH_{p-1}(z)\,\zeta,\zeta\rangle\le(p-1)(|z|-1)_+^{p-2}\,|\zeta|^{2}.
\end{align*} 
\end{lem}

\subsection{Difference Quotient}
\label{secquo}
We recall some properties of the finite difference quotient operator that will be needed in the sequel. 

\begin{dfn}
Let $F$ be a function defined in an open set $\Omega \subset \mathbb{R}^n$ and let $h \in \mathbb{R}^n$. We consider the function
$$ 
\tau_{h}F(x) :=F(x+h)-F(x) ,$$
\end{dfn}
\noindent which is defined in the set
$$\Omega_{|h|}: = \{ x \in \Omega : \mathrm{dist}(x,\partial \Omega)> |h|  \}.$$

We start with the description of some elementary properties that can be found, for example, in \cite[Chapter 8]{giusti}.
\begin{prop}\label{rapportoincrementale}
Let $F \in W^{1,p}(\Omega)$, with $p \geq1$, and let $G:\Omega \rightarrow \mathbb{R}$ be a measurable function.
Then
\\(i) $\tau_{h}F \in W^{1,p}(\Omega_{|h|})$ and 
$$D_{i}(\tau_{h}F)=\tau_{h}(D_{i}F).$$
(ii) If at least one of the functions $F$ or $G$ has support contained in $\Omega_{|h|}$, then
$$\displaystyle\int_{\Omega}F \tau_h G   dx = \displaystyle\int_{\Omega} G \tau_{-h}F dx.$$
(iii) We have $$\tau_{h} (FG)(x)= F(x+h)\tau_{h} G(x)+G(x) \tau_{h} F(x).$$
\end{prop}
The next result about the finite difference operator is a kind of integral version of Lagrange Theorem.
\begin{lem}\label{ldiff}
If $0<\rho<R,$ $|h|<\frac{R-\rho}{2},$ $1<p<+\infty$ and $F\in W^{1,p}(B_{R}, \mathbb{R}^k)$, then
\begin{center}
$\displaystyle\int_{B_{\rho}} |\tau_{h}F(x)|^{p} dx \leq c(n,p)|h|^{p} \displaystyle\int_{B_{R}} |DF(x)|^{p} dx$.
\end{center}
Moreover,
\begin{center}
$\displaystyle\int_{B_{\rho}} |F(x+h)|^{p} d x \leq  \displaystyle\int_{B_{R}} |F(x)|^{p}d x$.
\end{center}
\end{lem}

Finally, we recall the following fundamental result, whose
proof can be found in \cite[Lemma 8.2]{giusti}. 
\begin{lem}
\noindent \label{lem:RappIncre} Let $F:\mathbb{R}^{n}\rightarrow\mathbb{R}^{k}$,
$F\in L^{p}(B_{R},\mathbb{R}^{k})$ with $1<p<+\infty$. Suppose that
there exist $\rho\in(0,R)$ and a constant $M>0$ such that 
\[
\sum_{j=1}^{n}\int_{B_{\rho}}\left|\tau_{h}F(x)\right|^{p}dx\,\leq\,M^{p}\,\vert h\vert^{p}
\]
for every $h\in\mathbb{R}$ with $\vert h\vert<\frac{R-\rho}{2}$.
Then, $F\in W^{1,p}(B_{\rho},\mathbb{R}^{k})$ and 
\[
\Vert DF\Vert_{L^{p}(B_{\rho})}\leq M.
\]
\end{lem}

\subsection{Riesz Potential}
In this section, we recall some results about the Riesz potential. 
\\
If $f \in L^s(\Omega)$, with $s \ge 1$, and $0 < \alpha < n$, then the \textit{potential of order $\alpha$ generated by $f$} is defined by
$$U_\alpha[f](x)= \int_\Omega \dfrac{f(y)}{|x-y|^{n-\alpha}} \, dy.$$

\begin{prop}\label{rieszthm}
Let $f \in L^s(\Omega)$, with $s \ge 1$, and $0 < \alpha < n$. Let $r$ be defined by $$\frac{1}{r}=\frac{1}{s}-\frac{\alpha}{n}.$$
\begin{itemize}
    \item[(a)] If $1 < s < \frac{n}{\alpha}$, then there exists a positive constant $C=C(n,\alpha,s)$ such that for any $1 \le m \le r$ it holds
    \begin{equation}
        \Vert U_\alpha[f] \Vert_{L^m(\Omega)} \le \, C |\Omega|^{\frac{1}{m}-\frac{1}{r}} \Vert f \Vert_{L^s(\Omega)}. \label{Riesz}
    \end{equation}

    \item[(b)] If $s > \frac{n}{\alpha}$, then \eqref{Riesz} holds for any $m \le \infty$. 

    \item[(c)] If $s=\frac{n}{\alpha}$ (in this case $r=\infty$), then \eqref{Riesz} holds for any $m < \infty$.
\end{itemize}

\end{prop}

\section{Higher Differentiability }
The aim of this section is to provide some higher differentiability properties for the gradient of solutions to \eqref{Functional}. 
More precisely, we shall prove the following
\begin{thm}\label{Hp2}
Let $u \in W^{1,p}_{loc}(\Omega) \cap {L^{\infty}_{loc}(\Omega)}$ be a local weak solution of \eqref{Functional}. Assume that $f$ satisfies \eqref{F1} and \eqref{F2}. Then, we have
\begin{equation*}
    H_{\frac{p}{2}}(Du) \in W^{1,2}_{loc}(\Omega).
\end{equation*}
\end{thm}
\begin{proof}
    Fix a ball $B_R \Subset \Omega$ and consider a radius $\rho<R$, a positive cut-off function $\eta \in C_0^{\infty}(B_R)$, with $\eta=1$ on $B_\rho$, $0 \leq \eta \leq 1, \, |D\eta| \leq \frac{c}{R-\rho}.$ Let $|h| < \frac{R -\rho}{2}$. We test \eqref{Functional} with the function
$$\varphi = \tau_{ -h} \left(  \eta^2 \tau_{h} u \right) $$
thus obtaining
\begin{equation*}
    \int_{\Omega} \langle H_{p-1} ( Du), \tau_{-h} D  \left(  \eta^2 \tau_{h} u \right) \rangle \, dx =  \int_{\Omega} f(x,u) \tau_{ -h} \left(  \eta^2 \tau_{h} u \right) dx \, .
\end{equation*}
Hence by $(ii)$ of Proposition \ref{rapportoincrementale}
\begin{equation}
    \int_{\Omega} \langle \tau_{h} H_{p-1} ( Du), D  \left(  \eta^2 \tau_{h} u \right) \rangle \, dx =  \int_{\Omega} \tau_{ h} f(x,u)  \, \eta^2 \tau_{h} u \,  dx \, . \notag
\end{equation}
The previous equality can be written as follows
\begin{align}
 I_1 +I_2  &:=    \int_{\Omega} \langle \tau_{h} H_{p-1} ( Du),   \eta^2  \tau_{h} Du  \rangle \, dx + \int_{\Omega} \langle \tau_{h} H_{p-1} ( Du),   2 \eta D\eta \tau_{h} u  \rangle \, dx \notag \\  
   & = \, \int_{\Omega} \t_{ h} f(x,u)   \, \eta^2 \tau_{h} u \,  dx =: I_3. \notag
\end{align}
For the integral $I_3$ we have
\begin{align}
    I_3 = & \int_\Omega (f(x+h,u(x+h))-f(x,u(x+h))) \eta^2 \tau_h u \, dx \notag\\
    & + \int_\Omega (f(x,u(x+h))-f(x,u(x))) \eta^2 \tau_h u \, dx =: I_{3,1}+I_{3,2}. \notag
\end{align}
Therefore, we get
\begin{equation}\label{SommaInt}
   I_1 \leq |I_2| +|I_{3,1}|+|I_{3,2}| .
\end{equation}
Thanks to Lemma \ref{lemmaHalpha2}, we infer
\begin{align}\label{I_1}
  I_1 \geq {\dfrac{4}{p^2}} \int_\Omega \eta^2 |\tau_h H_{\frac{p}{2}}(Du)|^2 \,  dx.
\end{align}
By using the properties of $\eta$, Lemmas \ref{lemmaHalpha2} and \ref{ldiff}, Young's and H\"older's inequalities, we get
\begin{align}\label{I_2}
 |I_2| &\leq 2 \int_{\Omega} \eta |D \eta|  |\tau_h H_{p-1}(Du)| |\tau_{h} u| \,  dx \notag\\ \notag
 & \leq \frac{c}{R-\rho } \int_{\Omega} \eta  |\tau_h H_{p-1}(Du)| |\tau_{h} u| \, dx \\ \notag
  & \leq \frac{c}{R-\rho } \int_{\Omega} \eta  |\tau_h H_{\frac{p}{2}}(Du)|\left( |H_{\frac{p}{2}}(Du(x+h))|^{\frac{p-2}{p}} + |H_{\frac{p}{2}}(Du(x))|^{\frac{p-2}{p}} \right)  |\tau_{h} u| \,  dx \\ \notag
  & \leq \varepsilon \int_{\Omega} \eta^2  |\tau_h H_{\frac{p}{2}}(Du)|^2 \, dx \\ \notag 
  & \qquad +   \frac{c_\varepsilon}{(R-\rho)^2 } \int_{B_R} \left( |H_{\frac{p}{2}}(Du(x+h))|^{\frac{p-2}{p}} + |H_{\frac{p}{2}}(Du(x))|^{\frac{p-2}{p}} \right)^{2}  |\tau_{h} u|^2 \, dx \\ \notag
   & \leq \varepsilon \int_{\Omega} \eta^2  |\tau_h H_{\frac{p}{2}}(Du)|^2 \, dx \\ \notag 
  & \qquad +  \frac{c_\varepsilon}{(R-\rho)^2 } \left( \int_{B_{2R}}  |H_{\frac{p}{2}}(Du)|^{2} \, dx \right)^{\frac{p-2}{p}} \left( \int_{B_R}  |\tau_{h} u|^p \, dx\right)^{\frac{2}{p}} \\ \notag
   & \leq \varepsilon \int_{\Omega} \eta^2  |\tau_h H_{\frac{p}{2}}(Du)|^2 \, dx \\ \notag 
  & \qquad +  \frac{|h|^2 c_\varepsilon}{(R-\rho)^2 } \left( \int_{B_{2R}}  |Du|^{p} \, dx \right)^{\frac{p-2}{p}} \left( \int_{B_{2R}}  |D u|^p \, dx\right)^{\frac{2}{p}} \\ 
  & \leq \varepsilon \int_{\Omega} \eta^2  |\tau_h H_{\frac{p}{2}}(Du)|^2 \, dx +  \frac{|h|^2 c_\varepsilon}{(R-\rho)^2 }  \int_{B_{2R}}  |Du|^{p} \, dx .
\end{align}
We may use assumption \eqref{F2} and Lemma \ref{ldiff} to infer
\begin{align}\label{I_31}
   |I_{3,1}| \le & \int_\Omega \eta^2  | f(x+h,u(x+h))-f(x,u(x+h))| |\tau_h u| \, dx \notag\\
   \le & \, M |h|\int_{B_R} |\tau_h u| \,dx \notag\\
   \le & \, c(n,  M)  |h|^2 \int_{B_{2R}}  |D u| \,dx \notag\\
   \le & \, c(n, M)  |h|^2 {\int_{B_{2R}}  (1+|D u|^p) \,dx} .
\end{align}
 The local boundedness of $u$, assumption \eqref{F1} and Lemma \ref{ldiff} lead us to the following estimate for $I_{3,2}$.
\begin{align}\label{I_32}
   |I_{3,2}| \le & \int_\Omega \eta^2  | f(x,u(x+h))-f(x,u(x))| |\tau_h u| \, dx \notag\\
   \le & \, L \int_{B_R} |\tau_h u|^2 \,dx \notag\\
   \le & \, c(n,  L)  |h|^2 \int_{B_{2R}}  |D u|^2 \,dx \notag\\
    \le & \, c(n,  L)  |h|^2 \int_{B_{2R}} (1+ |D u|^p) \,dx.
\end{align}
 Inserting estimates \eqref{I_1}, \eqref{I_2}, \eqref{I_31} and \eqref{I_32} in \eqref{SommaInt}, we derive
 \begin{align*}
     {\dfrac{4}{p^2}} \int_\Omega \eta^2 |\tau_h H_{\frac{p}{2}}(Du)|^2 \, dx &\leq  \varepsilon \int_{\Omega} \eta^2  |\tau_h H_{\frac{p}{2}}(Du)|^2 \, dx  + \frac{|h|^2 c_\varepsilon}{(R-\rho)^2 }  \int_{B_{2R}}  |Du|^{p} \,dx \\
    & \qquad  + c(n, L )  |h|^2 \int_{B_{2R}}  (1+|D u|^p) \, dx .
 \end{align*}
Choosing $\varepsilon={\frac{2}{p^2}}$ and reabsorbing the first integral in the right-hand side into the left-hand side, we have
\begin{align*}
      \int_{B_\rho} |\tau_h H_{\frac{p}{2}}(Du)|^2 \, dx &\leq c|h|^2  \int_{B_{2R}} (1+ |D u|^p) \, dx ,
 \end{align*}
where also used that $\eta=1$ on $B_\rho $, with $c=c(p,n, M,L,R,\rho)$ positive constant. {Hence, by Lemma \ref{lem:RappIncre}, we get that $H_\frac{p}{2}(Du) \in W^{1,2}_{loc}(\Omega)$ and this concludes the proof.}
\end{proof}
A consequence of the previous theorem is a $W^{1,2}$-regularity of $H_{p-1}(Du)$.
\begin{cor}\label{cor1}
   Let $u \in  W^{1,\infty}_{loc}(\Omega)$ be a local weak solution of \eqref{Functional}. Assume that $f$ satisfies \eqref{F1} and \eqref{F2}. Then, we have
\begin{equation*}
    H_{p-1}(Du) \in W^{1,2}_{loc}(\Omega).
\end{equation*}
\end{cor}
\begin{proof}
    We note that 
    $$|DH_{\frac{p}{2}}(Du)|^2 \approx \left( |Du|-1 \right)_+^{p-2}|D^2u|^2 $$
and 
  $$|DH_{p-1}(Du)|^2 \approx \left( |Du|-1 \right)_+^{2p-4}|D^2u|^2. $$
  Since $Du \in L_{loc}^{\infty}(\Omega)$ and $p \ge 2$, we get
  \begin{equation*}
     |DH_{p-1}(Du)|^2 \le   c \,|DH_{\frac{p}{2}}(Du)|^2, 
  \end{equation*}
  hence from Theorem \ref{Hp2}, we have the desired conclusion.
    \end{proof}

\section{Second Order Regularity Results}\label{SecOrd}
In this section, we prove second order regularity results for solutions to \eqref{Functional}. More precisely, we have the following theorem.
\begin{thm}\label{mainthm}
   Let $u \in W^{1,\infty}_{loc}(\Omega)$ be a local weak solution of \eqref{Functional}. Assume that $f$ satisfies \eqref{F1} and \eqref{F2}. Then, for $0 \le \beta \le 1$ and $\gamma < n-2$ ($\gamma=0$ if $n=2$), the following estimates
    \begin{align}
    & \sup_{x \in \Omega} \int_{  B_\rho \cap \{|u_{x_j}|>1\}}\frac{(|Du|-1)_+^{p-1} \,|Du_{x_j}|^2}{(|u_{x_j}|-1)_+^\beta \, |x-y|^{\gamma}}   \, dy \le \, C \label{mainest1}
\end{align}
and
\begin{equation}
     \sup_{x \in \Omega }\int_{ B_\rho \cap \{|Du|>1\}} \frac{(|Du|-1)_+^{p-1-\beta} \,|D^2u|^2}{ |x-y|^{\gamma}}  \, dy \le \, C \label{mainest2}
\end{equation}
hold for every ball $B_\rho \Subset \Omega $, where $C=C(\rho,n,p,\beta,\gamma, M,L, \Vert u \Vert_{W^{1,\infty}})$ is a positive constant.
\end{thm}
\begin{proof}
{
By virtue of Corollary \ref{cor1}, we can differentiate the equation in \eqref{Functional} with respect to $x_{j}$ for some $j\in\{1,\dots,n\}$
and then we integrate by parts, thus obtaining 
\begin{equation}
\int_{\Omega}\langle DH_{p-1}(Du)\,Du_{x_j},D\psi\rangle\,dx\,=
\int_{\Omega}f_{x_j}(x,u)\,\psi\,dx \, + \,
\int_{\Omega}f_u(x,u)\,u_{x_j}\,\psi\,dx,\qquad\forall\,\psi\in W_{0}^{1,p}(\Omega).\label{secondvar}
\end{equation}   
Let $G_\varepsilon : \mathbb{R} \to \mathbb{R} $ be defined as {
\begin{align}
G_\varepsilon (s)=
\begin{cases}
   0 \qquad & 0 \le s \le 1+ \varepsilon \\
  2s-2(1+\varepsilon) \qquad & 1+ \varepsilon < s < 1+ 2 \varepsilon \\
   s-1 \qquad & s \ge 1+ 2 \varepsilon 
\end{cases}
\end{align}
and $G_\varepsilon(s)=-G_\varepsilon(-s)$ for $s \le 0$.
Moreover, let  $K_\delta : \mathbb{R} \to \mathbb{R} $ be defined as
\begin{align}\label{Kappa}
K_\delta (s)=
\begin{cases}
   0 \qquad & 0\le s \le \delta \\
  2s- 2 \delta  \qquad & \delta < s < 2 \delta \\
   s \qquad & s \ge 2 \delta 
\end{cases}
\end{align}
and $K_\delta(s)=-K_\delta(-s)$ for $s \le 0$.}

  Fix a ball $B_R \Subset \Omega$ and consider a radius $0 <\rho<R$, a positive cut-off function $\eta \in C_0^{\infty}(B_R)$, with $\eta=1$ on $B_\rho$, $0 \leq \eta \leq 1, \, |D\eta| \leq \frac{c}{R-\rho}.$ Set 
  $$\psi (y)= \frac{G_\varepsilon(u_{x_j})}{(|u_{x_j}|-1)_+^\beta}  \frac{K_\delta (|x-y|)}{|x-y|^{\gamma +1}} \eta^2 (y).$$
We compute
\begin{align*}
    D_y \psi (y)&= \frac{G'_\varepsilon(u_{x_j})}{(|u_{x_j}|-1)_+^\beta} {Du_{x_j}}\frac{K_\delta (|x-y|)}{|x-y|^{\gamma +1}} \eta^2 (y) \\
    & \qquad- \beta \frac{G_\varepsilon(u_{x_j})}{(|u_{x_j}|-1)_+^{\beta+1}} \frac{|u_{x_j}|}{u_{x_j}} {Du_{x_j}}\frac{K_\delta (|x-y|)}{|x-y|^{\gamma +1}} \eta^2 (y) \\
     & \qquad + \frac{G_\varepsilon(u_{x_j})}{(|u_{x_j}|-1)_+^\beta}  \frac{K_\delta (|x-y|)}{|x-y|^{\gamma +1}} 2\eta (y) D \eta(y) \\
     & \qquad + \frac{G_\varepsilon(u_{x_j})}{(|u_{x_j}|-1)_+^\beta}  D_y \left(\frac{K_\delta (|x-y|)}{|x-y|^{\gamma +1}} \right) \eta^2 (y).
\end{align*}
Using $\psi$ as test function in \eqref{secondvar}, we get
\begin{align}\label{SV}
   & \int_{\Omega}\langle DH_{p-1}(Du)\,Du_{x_j},D u_{x_j}\rangle  \frac{T_\varepsilon(u_{x_j})}{(|u_{x_j}|-1)_+^\beta}  \frac{K_\delta (|x-y|)}{|x-y|^{\gamma +1}} \eta^2 (y) \,dy\, \notag \\
   & \qquad +2 \int_{\Omega}\langle DH_{p-1}(Du)\,Du_{x_j},D \eta \rangle  \frac{G_\varepsilon(u_{x_j})}{(|u_{x_j}|-1)_+^\beta}  \frac{K_\delta (|x-y|)}{|x-y|^{\gamma +1}} \eta (y) \,dy\, \notag \\
    & \qquad +\int_{\Omega}\langle DH_{p-1}(Du)\,Du_{x_j},D \left(\frac{K_\delta (|x-y|)}{|x-y|^{\gamma +1}} \right) \rangle  \frac{G_\varepsilon(u_{x_j})}{(|u_{x_j}|-1)_+^\beta}   \eta^2 (y) \,dy\, \notag \\
   &=\,\int_{\Omega}f_{x_j}(x,u)\,\frac{G_\varepsilon(u_{x_j})}{(|u_{x_j}|-1)_+^\beta} \frac{K_\delta (|x-y|)}{|x-y|^{\gamma +1}} \eta^2 (y)\,dy \notag\\
   & +\,\int_{\Omega}f_u(x,u)\,u_{x_j}\,\frac{G_\varepsilon(u_{x_j})}{(|u_{x_j}|-1)_+^\beta} \frac{K_\delta (|x-y|)}{|x-y|^{\gamma +1}} \eta^2 (y)\,dy  ,
\end{align}
where we denoted
\begin{equation}\label{Te}
    T_\varepsilon (s)= G_\varepsilon'(s)-\beta \frac{G_\varepsilon(s)}{s} \frac{|s|}{(|s|-1)_+}.
\end{equation}
An easy calculation shows that, for $\beta \in [0,1]$, it holds $T_\varepsilon(s)\ge 0.$ 

We write the equality \eqref{SV} as follows
$$I_1 + 2 I_2 + I_3 = I_4+I_5,$$
which yields
\begin{equation}\label{Somma2INt}
    I_1 \le  2 |I_2| + |I_3| + |I_4|+|I_5|.
\end{equation}
Consider the term $I_2$. By using Cauchy-Schwartz's and Young's inequalities, we have
\begin{align}
    2|I_2| &\le 2 \int_{\Omega}|\langle DH_{p-1}(Du)\,Du_{x_j},D \eta \rangle  |\frac{|G_\varepsilon(u_{x_j})|}{(|u_{x_j}|-1)_+^\beta} \frac{|K_\delta (|x-y|)|}{|x-y|^{\gamma +1}} \eta (y) \,dy\, \notag \\
    &  \le 2 \int_{\Omega} \sqrt{\langle DH_{p-1}(Du)\,Du_{x_j},D u_{x_j} \rangle  } \sqrt{\langle DH_{p-1}(Du)\,D_\eta,D \eta \rangle  } \cdot \notag \\
    & \qquad \cdot \frac{|G_\varepsilon(u_{x_j})|}{(|u_{x_j}|-1)_+^\beta} \frac{|K_\delta (|x-y|)|}{|x-y|^{\gamma +1}} \eta (y) \,dy\, \notag \\
    & \le \sigma \int_{\Omega} \langle DH_{p-1}(Du)\,Du_{x_j},D u_{x_j} \rangle \frac{|G_\varepsilon(u_{x_j})|}{(|u_{x_j}|-1)_+^\beta}  \frac{1}{|u_{x_j}|}\frac{|K_\delta (|x-y|)|}{|x-y|^{\gamma +1}} \eta^2 (y) \,dy\,   \notag \\
    & \qquad + c_\sigma \int_{\Omega} \langle DH_{p-1}(Du)\,D \eta ,D \eta \rangle \frac{|G_\varepsilon(u_{x_j})|}{(|u_{x_j}|-1)_+^\beta} |u_{x_j}|\frac{|K_\delta (|x-y|)|}{|x-y|^{\gamma +1}}  \, dy.  \notag 
\end{align}
Using Lemma \ref{lemH}, we infer
\begin{align}
   2|I_2| &\le \sigma  \int_{\Omega} \langle DH_{p-1}(Du)\,Du_{x_j},D u_{x_j} \rangle \frac{|G_\varepsilon(u_{x_j})|}{(|u_{x_j}|-1)_+^\beta}  \frac{1}{|u_{x_j}|}\frac{|K_\delta (|x-y|)|}{|x-y|^{\gamma +1}} \eta^2 (y) \,dy\,   \notag \\
    & \qquad + c_\sigma \int_{\Omega} (|Du|-1)^{p-2}_+ |D \eta|^2 \frac{|u_{x_j}|(|u_{x_j}|-1)_+}{(|u_{x_j}|-1)_+^\beta}  \frac{|K_\delta (|x-y|)|}{|x-y|^{\gamma +1}} dy \notag \\
    &\le \sigma  \int_{\Omega} \langle DH_{p-1}(Du)\,Du_{x_j},D u_{x_j} \rangle \frac{|G_\varepsilon(u_{x_j})|}{(|u_{x_j}|-1)_+^\beta}  \frac{1}{|u_{x_j}|}\frac{|K_\delta (|x-y|)|}{|x-y|^{\gamma +1}} \eta^2 (y) \,dy\,   \notag \\
    & \qquad + c_\sigma \int_{B_R}   \frac{|K_\delta (|x-y|)|}{|x-y|^{\gamma +1}} \,dy,  \label{2I2}
\end{align}
where we also used that $u \in W^{1, \infty},$ the properties of $\eta$ and the fact that $|G_\varepsilon(s)| \lesssim (|s|-1)_+.$

In order to estimate $I_3$, we notice that 
\begin{align}\label{Dy1}
     D \left(\frac{K_\delta (|x-y|)}{|x-y|^{\gamma +1}} \right)  = \left( K'_\delta (|x-y|)-(\gamma+1) \frac{K_\delta (|x-y|)}{|x-y|}\right)\frac{D_y (|x-y|)}{|x-y|^{\gamma +1}}
\end{align}
and 
\begin{equation}\label{Dy2}
    \left\lvert K'_\delta (|x-y|)-(\gamma+1) \frac{K_\delta (|x-y|)}{|x-y|} \right\rvert \le c(\gamma),
\end{equation}
since $ K_\delta$ is Lipschitz continuous such that $0 \le K_\delta'(s) \le 2$ and $ |K_\delta(s)| \lesssim |s|$. 

Arguing as for the integral $I_2$ and using \eqref{Dy1} and \eqref{Dy2}, we obtain
\begin{align}\label{2I3}
|I_3| &\le c(\gamma) \int_{\Omega}|\langle DH_{p-1}(Du)\,Du_{x_j},D_y (|x-y|) \rangle  |\frac{|G_\varepsilon(u_{x_j})|}{(|u_{x_j}|-1)_+^\beta}  \frac{1}{|x-y|^{\gamma +1}} \eta^2 (y) \,dy\, \notag \\
    & \le \sigma \int_{\Omega} \langle DH_{p-1}(Du)\,Du_{x_j},D u_{x_j} \rangle \frac{|G_\varepsilon(u_{x_j})|}{(|u_{x_j}|-1)_+^\beta}  \frac{1}{|u_{x_j}|}\frac{1}{|x-y|^{\gamma}} \eta^2 (y) \,dy\,   \notag \\
    & \qquad + c_\sigma \int_{\Omega} \langle DH_{p-1}(Du)\,D_y (|x-y|) ,D_y (|x-y|) \rangle \frac{|G_\varepsilon(u_{x_j})|}{(|u_{x_j}|-1)_+^\beta} |u_{x_j}|\frac{1}{|x-y|^{\gamma +2}} \eta^2(y) \, dy \notag \\
    & \le \sigma \int_{\Omega} \langle DH_{p-1}(Du)\,Du_{x_j},D u_{x_j} \rangle \frac{|G_\varepsilon(u_{x_j})|}{(|u_{x_j}|-1)_+^\beta}  \frac{1}{|u_{x_j}|}\frac{1}{|x-y|^{\gamma}} \eta^2 (y) \,dy\,   \notag \\
    & \qquad + c_\sigma \int_{\Omega} (|Du|-1)^{p-2}_+  \frac{|u_{x_j}|(|u_{x_j}|-1)_+}{(|u_{x_j}|-1)_+^\beta}  \frac{\eta^2 (y)|}{|x-y|^{\gamma +2}}\, dy \notag \\
    &\le \sigma  \int_{\Omega} \langle DH_{p-1}(Du)\,Du_{x_j},D u_{x_j} \rangle \frac{|G_\varepsilon(u_{x_j})|}{(|u_{x_j}|-1)_+^\beta}  \frac{1}{|u_{x_j}|}\frac{1}{|x-y|^{\gamma}} \eta^2 (y) \,dy\,   \notag \\
    & \qquad + c_\sigma \int_{B_R}   \frac{1}{|x-y|^{\gamma +2}}\, dy. 
\end{align}
Using the property of $\eta$ and the Lipschitz continuity of $f$ with respect to the $x$-variable, we derive
\begin{align}
  |I_4| &\le \int_{\Omega} |f_{x_j}(x,u)|  \frac{|G_\varepsilon(u_{x_j})|}{(|u_{x_j}|-1)_+^\beta} \frac{K_\delta (|x-y|)}{|x-y|^{\gamma +1}} \eta^2 (y) \, dy \notag \\
  & \le c M\int_{\Omega} (|u_{x_j}|-1)_+^{-\beta+1}
  \frac{K_\delta (|x-y|)}{|x-y|^{\gamma +1}} \eta^2 (y) \, dy \notag\\
  & \le c M\int_{B_R}
  \frac{K_\delta (|x-y|)}{|x-y|^{\gamma +1}} \, dy,\label{I_4}
\end{align}
where we have used that $\beta \le 1$ and $|G_\varepsilon(s)| \lesssim (|s|-1)_+$.
The integral $I_5$ can be estimate as follows
\begin{align}
  |I_5| &\le \int_{\Omega} |f_u(x,u)| |u_{x_j}| \frac{|G_\varepsilon(u_{x_j})|}{(|u_{x_j}|-1)_+^\beta} \frac{K_\delta (|x-y|)}{|x-y|^{\gamma +1}} \eta^2 (y) \, dy \notag \\
  & \le cL\int_{\Omega} (|u_{x_j}|-1)_+^{-\beta+1}
  \frac{K_\delta (|x-y|)}{|x-y|^{\gamma +1}} \eta^2 (y) \, dy \notag\\
  & \le cL\int_{B_R}
  \frac{K_\delta (|x-y|)}{|x-y|^{\gamma +1}} \, dy, \label{I_5}
\end{align}
where we have used that $|G_\varepsilon(s)| \lesssim (|s|-1)_+$, the property of $\eta$, $\beta \le 1$, assumption \eqref{F1} {and the Lipschitz continuity of $u$.}

Inserting \eqref{2I2}, \eqref{2I3}, \eqref{I_4} and \eqref{I_5} in \eqref{Somma2INt}, we derive
\begin{align}\label{StimaF}
    & \int_{\Omega}\langle DH_{p-1}(Du)\,Du_{x_j},D u_{x_j}\rangle  \frac{T_\varepsilon(u_{x_j})}{(|u_{x_j}|-1)_+^\beta} \frac{K_\delta (|x-y|)}{|x-y|^{\gamma +1}} \eta^2 (y) \,dy\, \notag \\
     &\le \sigma  \int_{\Omega} \langle DH_{p-1}(Du)\,Du_{x_j},D u_{x_j} \rangle \frac{|G_\varepsilon(u_{x_j})|}{(|u_{x_j}|-1)_+^\beta} \frac{1}{|u_{x_j}|}\frac{|K_\delta (|x-y|)|}{|x-y|^{\gamma +1}} \eta^2 (y) \,dy\,   \notag \\
     & \quad + \sigma  \int_{\Omega} \langle DH_{p-1}(Du)\,Du_{x_j},D u_{x_j} \rangle \frac{|G_\varepsilon(u_{x_j})|}{(|u_{x_j}|-1)_+^\beta}  \frac{1}{|u_{x_j}|}\frac{1}{|x-y|^{\gamma}} \eta^2 (y) \,dy\,   \notag \\
    & \quad + c_\sigma \int_{B_R}   \frac{1}{|x-y|^{\gamma +2}} \,dy + c_\sigma \int_{B_R}   \frac{|K_\delta (|x-y|)|}{|x-y|^{\gamma +1}} \,dy.
\end{align}
We know that $\int_{B_R}|x-y|^{-q} \, dy$ is uniformly bounded for any $q <n$. In particular, since $\gamma < n-2$, this is true for $q=\gamma +1$ and $q=\gamma +2$. Therefore, for any fixed $\varepsilon >0$, since $K_\delta(s)$ converges to $s$ as $\delta \to 0^+$, by dominated convergence theorem, letting $\delta \to 0^+$ in \eqref{StimaF}, we infer
\begin{align*}
    & \int_{\Omega}\langle DH_{p-1}(Du)\,Du_{x_j},D u_{x_j}\rangle  \frac{T_\varepsilon(u_{x_j})}{(|u_{x_j}|-1)_+^\beta}  \frac{1}{|x-y|^{\gamma}} \eta^2 (y) \,dy\, \notag \\
     &\le 2\sigma  \int_{\Omega} \langle DH_{p-1}(Du)\,Du_{x_j},D u_{x_j} \rangle \frac{|G_\varepsilon(u_{x_j})|}{(|u_{x_j}|-1)_+^\beta}  \frac{1}{|u_{x_j}|}\frac{1}{|x-y|^{\gamma}} \eta^2 (y) \,dy\,   \notag \\
    & \quad + c_\sigma \int_{B_R}   \frac{1}{|x-y|^{\gamma +2}} \,dy + c_\sigma \int_{B_R}   \frac{1}{|x-y|^{\gamma }} \,dy.
\end{align*}
Hence, reabsorbing the first integral in the right-hand side into the left-hand side, recalling the definition of the function $T_\varepsilon$ in \eqref{Te} and using Lemma \ref{lemH}, we get
\begin{align}\label{eq:int}
    & \int_{\Omega}\frac{(|Du|-1)_+^{p-1} \,|Du_{x_j}|^2}{|Du|\,(|u_{x_j}|-1)_+^\beta \, |x-y|^{\gamma}}  \notag\\
     & \qquad \cdot
     \left(  G_\varepsilon'(u_{x_j})-\beta \frac{G_\varepsilon(u_{x_j})}{u_{x_j}} \frac{|u_{x_j}|}{(|u_{x_j}|-1)_+}- 2 \sigma \frac{G_\varepsilon(u_{x_j})}{u_{x_j}} \right)
     \eta^2 (y) \,dy\ \le \,C . 
\end{align}
Now, choosing $0 <\sigma < \frac{1-\beta}{2}$, we find that
$$G_\varepsilon'(u_{x_j})-\beta \frac{G_\varepsilon(u_{x_j})}{u_{x_j}} \frac{|u_{x_j}|}{(|u_{x_j}|-1)_+}- 2 \sigma \frac{G_\varepsilon(u_{x_j})}{u_{x_j}}>0.$$
Since, for $|u_{x_j}|>1$, as $\varepsilon \to 0^+ $
$$ G_\varepsilon'(u_{x_j})-\beta \frac{G_\varepsilon(u_{x_j})}{u_{x_j}} \frac{|u_{x_j}|}{(|u_{x_j}|-1)_+}- 2 \sigma \frac{G_\varepsilon(u_{x_j})}{u_{x_j}} \to 1-\beta -2 \sigma \frac{(|u_{x_j}|-1)_+}{|u_{x_j}|}. $$
Letting $\varepsilon \to 0^+$ in \eqref{eq:int}, by Fatou's Lemma 
 and using the boundedness of $Du$, we infer
\begin{align}
    & \int_{B_\rho \cap \{|u_{x_j}|>1\}}\frac{(|Du|-1)_+^{p-1} \,|Du_{x_j}|^2}{(|u_{x_j}|-1)_+^\beta \, |x-y|^{\gamma}}  \eta^2(y) \, dy \le \, C, \notag
\end{align}
for a positive constant $C=C(\rho,R,n,p,\beta,\alpha,\gamma,M,L, \Vert u \Vert_{W^{1,\infty}})$ independent of $x$. Therefore, using the fact that $\eta =1$ on $B_\rho$, we have
\begin{align}
    & \sup_{x \in \Omega} \int_{  B_\rho \cap \{|u_{x_j}|>1\}}\frac{(|Du|-1)_+^{p-1} \,|Du_{x_j}|^2}{(|u_{x_j}|-1)_+^\beta \, |x-y|^{\gamma}}   \, dy \le \, C. \notag
\end{align}
In particular, since $|Du| \ge |u_{x_j}|$ and $0 \le \beta \le 1$, we can conclude that
\begin{equation*}
     \sup_{x \in \Omega }\int_{ B_\rho \cap \{|Du|>1\}}\frac{(|Du|-1)_+^{p-1-\beta} \,|D^2u|^2}{ |x-y|^{\gamma}}  \, dy \le \, C. \qedhere
\end{equation*}
}
\end{proof}

\section{Weak Comparison Principle}

{We point out that the second order regularity results proved in Section \ref{SecOrd} hold without any sign assumption on the right-hand side $f$. Now, assuming that $f$ has a sign, as a consequence of Theorem \ref{mainthm}, we obtain the summability properties of $\frac{1}{(|Du|-1)_+}$.}

\begin{prop}\label{invp}
 Let $u \in W^{1,\infty}_{loc}(\Omega)$ be a local weak solution of \eqref{Functional}. Assume that $f$ {is a positive function } satisfying \eqref{F1} and \eqref{F2}. For $\gamma < n-2$ ($\gamma=0$ if $n=2$) and $t \in [0,\frac{p-2}{p}]$, then the following estimate
\begin{align}
 {  \sup_{x \in \Omega} \int_{\Omega'}  \frac{1}{(|Du|-1)_+^{pt}}  \frac{1}{|x-y|^{\gamma }}  \, dy \le C.} \notag
\end{align}
 holds for every {open set $\Omega' \Subset \Omega $}, where $C=C(n,p,t,\gamma, M,L,\Vert u \Vert_{W^{1,\infty}},\Omega')$ is a positive constant. {Moreover, we have that $$|\{x \in \Omega' \,:\,|Du|\le 1 \}| =0.$$}
\end{prop}

\begin{proof}
Fix $x \in \Omega$ and a ball $B_R \Subset \Omega$ and consider a radius $\rho<R$, a positive cut-off function $\eta \in C_0^{\infty}(B_R)$, with $\eta=1$ on $B_\rho$, $0 \leq \eta \leq 1, \, |D\eta| \leq \frac{c}{R-\rho}.$ Let $K_\delta$ be the function defined in \eqref{Kappa}. For $t=\frac{p-3+\beta}{p} \in [0,\frac{p-2}{p}]$, with $\beta \in [0,1]$, we set 
  \begin{equation}
      \psi (y)= \frac{1}{\left[ \left(|Du|-1 \right)_{+}+\varepsilon \right]^{pt}}  \frac{K_\delta (|x-y|)}{|x-y|^{\gamma +1}} \eta^2 (y). \label{psitest}
  \end{equation}
We note that, since the map $y \mapsto f(y,u(y))$ is continuous and non negative in $B_R$, we have 
\begin{equation}
    f(y,u(y)) \ge c_1>0, \qquad \text{for all} \ y \in B_R. \label{fpositiva}
\end{equation}
Using \eqref{fpositiva} and testing the equation \eqref{Weaksol} with the function $\psi$ defined in \eqref{psitest}, we get
\begin{align}
    c_1 \int_{B_R} \psi(y) \, dy & \le  \int_{B_R} \psi(y)f(y,u(y)) \, dy \notag\\
    & \le  \int_{B_R} \langle H_{p-1}(Du), D \psi(y) \rangle \, dy \notag\\
    & \le -pt \int_{B_R} \langle H_{p-1}(Du), \frac{Du}{|Du|} D^2u \rangle  \left[ \left(|Du|-1 \right)_{+}+\varepsilon \right]^{-pt-1}  \frac{K_\delta (|x-y|)}{|x-y|^{\gamma +1}} \eta^2 (y)\, dy \notag\\
    & \, \quad + 2 \int_{B_R} \langle H_{p-1}(Du), D_y\eta(y)\rangle  \frac{1}{\left[ \left(|Du|-1 \right)_{+}+\varepsilon \right]^{pt}}  \frac{K_\delta (|x-y|)}{|x-y|^{\gamma +1}} \eta (y)  \, dy \notag\\
    & \, \quad + \int_{B_R} \langle H_{p-1}(Du),  D_y \left(\frac{K_\delta (|x-y|)}{|x-y|^{\gamma +1}} \right)  \rangle \frac{1}{\left[ \left(|Du|-1 \right)_{+}+\varepsilon \right]^{pt}} \eta^2(y) \,dy \notag \\
    & = : I_1+I_2+I_3 .\label{eq:prop1} 
\end{align}
Now, we estimate the integral $I_1$. Applying Young's inequality, recalling that $t=\frac{p-3+\beta}{p}$, with $\beta \in [0,1]$, and using estimate \eqref{mainest2}, we infer
\begin{align}
    |I_1| &\le \int_{B_R} (|Du|-1)_+^{p-1}\left[ \left(|Du|-1 \right)_{+}+\varepsilon \right]^{-pt-1} |D^2u|  \frac{1}{|x-y|^\gamma} \eta^2(y) \, dy \notag\\
    & \le \theta \int_{B_R} \frac{1}{\left[ \left(|Du|-1 \right)_{+}+\varepsilon \right]^{pt}}  \frac{1}{|x-y|^\gamma} \eta^2(y) \, dy \notag\\
    & \, \quad + C_\theta\int_{B_R} { \left(|Du|-1 \right)_{+}^{2p-4-pt}} |D^2u|^2 \frac{1}{|x-y|^\gamma} \eta^2(y) \, dy \notag\\
     & {\le \theta \int_{B_R} \frac{1}{\left[ \left(|Du|-1 \right)_{+}+\varepsilon \right]^{pt}}  \frac{1}{|x-y|^\gamma} \eta^2(y) \, dy} \notag\\
   &{  \, \quad + C_\theta\int_{B_R} { \left(|Du|-1 \right)_{+}^{p-1-\beta}} |D^2u|^2 \frac{1}{|x-y|^\gamma} \eta^2(y) \, dy }\notag\\
    &  \le \theta \int_{B_R} \frac{1}{\left[ \left(|Du|-1 \right)_{+}+\varepsilon \right]^{pt}}  \frac{1}{|x-y|^\gamma} \eta^2(y) \, dy \,+ \,C_\theta,
    \label{i_1}
\end{align}
where we also used that $|K_\delta(s)| \le |s|$ {and $\theta \in (0,1)$ will be chosen later.}

Let us consider the term $I_2$. Since $t \leq \frac{p-2}{p}$, we have
\begin{equation}\label{StimaGradient}
    \frac{(|Du|-1)^{p-1}_+}{\left[ \left(|Du|-1 \right)_{+}+\varepsilon \right]^{pt}} \le \frac{[(|Du|-1)_+ + \varepsilon ]^{p-1}}{\left[ \left(|Du|-1 \right)_{+}+\varepsilon \right]^{pt}} \le c(\Vert Du \Vert_{L^{\infty}}).
    \end{equation}
Therefore, we obtain
\begin{align}
    |I_2| & \le 2 \int_{B_R} \frac{(|Du|-1)_+^{p-1}}{\left[ \left(|Du|-1 \right)_{+}+\varepsilon \right]^{pt}} |D \eta| \frac{1}{|x-y|^\gamma} \eta(y) \, dy \notag\\
    & \le c\int_{B_R} \frac{1}{|x-y|^\gamma} \, dy =:c_2 .\label{i_2'}
\end{align}
 We can estimate $I_3$ similarly as for $I_2$. Recalling \eqref{Dy1}, \eqref{Dy2} and \eqref{StimaGradient}, it holds
\begin{align}
    |I_3| & \le c \int_{B_R} \frac{(|Du|-1)_+^{p-1}}{\left[ \left(|Du|-1 \right)_{+}+\varepsilon \right]^{pt}}\frac{1}{|x-y|^\gamma} \eta^2(y) \, dy \notag\\
    & \le c\int_{B_R} \frac{1}{|x-y|^\gamma} \, dy =:c_3 .\label{i_3}
\end{align}
Inserting estimates \eqref{i_1}, \eqref{i_2'} and \eqref{i_3} in \eqref{eq:prop1}, we find that
\begin{align*}
   & c_1 \int_{B_R}  \frac{1}{\left[ \left(|Du|-1 \right)_{+}+\varepsilon \right]^{pt}}  \frac{K_\delta (|x-y|)}{|x-y|^{\gamma +1}} \eta^2 (y) \, dy \\
    & \le \theta \int_{B_R}  \frac{1}{\left[ \left(|Du|-1 \right)_{+}+\varepsilon \right]^{pt}}  \frac{1}{|x-y|^\gamma} \eta^2(y) \, dy \,+ \,C_\theta \,+\,c_2\,+\,c_3.
\end{align*}
Letting $\delta \to 0^+$, by Dominated Convergence Theorem, we get
\begin{align*}
   & c_1 \int_{B_R}  \frac{1}{\left[ \left(|Du|-1 \right)_{+}+\varepsilon \right]^{pt}}  \frac{1}{|x-y|^{\gamma }} \eta^2 (y) \, dy \\
    & \le \theta \int_{B_R}  \frac{1}{\left[ \left(|Du|-1 \right)_{+}+\varepsilon \right]^{pt}}  \frac{1}{|x-y|^\gamma} \eta^2(y) \, dy \,+ \,C_\theta \,+\,c_2\,+\,c_3.
\end{align*}
Hence, choosing $\theta =\frac{c_1}{2}$, we can reabsorb the first integral in the right-hand side into the left-hand side and, recalling that $\eta=1$ on $B_\rho$, we get
\begin{align}
   & \int_{B_\rho}  \frac{1}{\left[ \left(|Du|-1 \right)_{+}+\varepsilon \right]^{pt}}  \frac{1}{|x-y|^{\gamma }} \, dy \le C, \label{eq:prop2}
\end{align}
{
which yields
\begin{equation}
    \dfrac{|B_\rho \cap \{|Du| \le 1 \}|}{\varepsilon^{pt}} \le C, \label{meas}
\end{equation}
where the constant $C$ is independent of both $\varepsilon$ and $x$.
}
Eventually, letting $\varepsilon \to 0^+$ in \eqref{eq:prop2} and \eqref{meas} and then applying a standard covering argument, we complete the proof.
\end{proof}

{In Proposition \ref{invp}, we have proved that the measure of the degeneracy set of the equation \eqref{Functional} is zero in every open set $\Omega'$ compactly supported in $\Omega$. This allows us to establish a weak comparison principle, in the sense that  if $u ,v$ are two local weak solutions of \eqref{Functional} such that $u \le v$ on $\partial \Omega'$,  then $u \le v$ in $\Omega'$. More precisely, we have the following. }

\begin{thm}
     Let $u ,v\in W^{1,\infty}_{loc}(\Omega)$ be local weak solutions of \eqref{Functional}. Assume that $f$ {is a positive function } satisfying \eqref{F1} and \eqref{F2}. Moreover, we suppose that for every $x \in \Omega$
     \begin{equation}
         s \mapsto f(x,s) \quad \textit{is non-increasing}. \label{decrescente}
     \end{equation}
     If $u \le v$ on $\partial \Omega'$, for an open set $\Omega' \Subset \Omega$, then $u \le v$ in $\Omega'$. Moreover, if 
     \begin{equation}
         \{ x \in \Omega : |Du| \le 1 \} \Subset \Omega, \label{degset}
     \end{equation}
     we can choose $\Omega' = \Omega$.
\end{thm}
\proof Let us consider $w=(u-v)_+$. Since $w \in W^{1,p}_0(\Omega')$, it can be used as test function in \eqref{Weaksol}. Thanks to \eqref{decrescente}, we obtain
\begin{align}
    \int_{\Omega' \cap \{ u>v \}} \langle  H_{p-1}(Du)-H_{p-1}(Dv), Dw\rangle \, dx = \int_{\Omega'\cap \{ u>v \}} (f(x,u)-f(x,v))w \, dx \le 0. \label{stimaconf}
\end{align}
The integral in the left-hand side of \eqref{stimaconf} can be estimated by using inequality \eqref{monprop}, thus getting
\begin{align}
    &\int_{\Omega' \cap \{ u>v \}} |Du-Dv|^2 \left[ \frac{(|Du|-1)_+^p}{|Du|^2} +  \frac{(|Dv|-1)_+^p}{|Dv|^2} \right] \, dx \notag\\
    &\le \int_{\Omega' \cap \{ u>v \}} \langle  H_{p-1}(Du)-H_{p-1}(Dv), Dw\rangle \, dx \le 0 .\notag
\end{align}
Since $$|\{ x \in \Omega' \, : \,|Du| \le 1 \}|=0=|\{ x \in \Omega' \, : \,|Dv| \le 1 \}|$$ by virtue of Proposition \ref{invp}, then we have $Du=Dv$ a.e.\ in $\Omega' \cap \{ u>v \}$, which yields $w=(u-v)_+=0$ in $\Omega' \cap \{ u>v \}$ and so $|\{ u >v \}|=0$, i.e.\ $u \le v$ in $\Omega'$.

\noindent
{Now, if \eqref{degset} is in force, then, by Proposition \ref{invp}, we have that
$$ |\{ x \in \Omega \, : \,|Du| \le 1 \}|=0,$$
which implies that the comparison principle holds in the whole $\Omega$.}
\endproof

\section{Weighted Sobolev Inequality}
In the next theorem we state a weighted Sobolev inequality. 

\begin{thm}\label{WS}
    Let $\Omega $ be a bounded open set. Let $\rho \in L^\infty(\Omega)$ be a non-negative weight such that
    \begin{equation}
        \int_\Omega \frac{1}{ \rho^t \, |x-y|^\gamma} \, dy \le  K,
    \end{equation}
    for some $t >0$ and $\gamma < n$, where $K$ is a positive constant independent of $x$. Assume that $q$ is an exponent satisfying the following conditions:
    \begin{itemize}
        \item[(i)] $q > \frac{n-\gamma}{t}$,
        \item[(ii)] $q >1+\frac{1}{t}$,
       \item[(iii)] $q< \frac{n-\gamma}{t}+n.$ 
    \end{itemize}
    Let $q^*$ be defined by
    \begin{equation}\label{q*}
        \dfrac{1}{q^*}=\dfrac{1}{q} \left( 1+\dfrac{1}{t}-\dfrac{\gamma}{nt} \right)-\dfrac{1}{n}.
    \end{equation}
Then, there exists a positive constant $c=c(n,q,\rho,t,\gamma, K)$   such that for any $u \in W^{1,q}_0(\Omega,\rho)$ it holds
\begin{equation}
    \Vert u \Vert_{L^{q^*}(\Omega)} \le c \, \Vert Du \Vert_{L^{q}(\Omega, \rho)}.
\end{equation}
\end{thm}
\proof
By a standard density argument, we may assume $u \in \mathcal{C}^1_0(\Omega)$. Therefore, there exists a constant $C_n$, depending only on $n$, such that for any $x \in \Omega$, we have
\begin{align}
    |u(x)|\le & \, C_n \int_\Omega \dfrac{|Du(y)|}{|x-y|^{n-1}} \, dy  \notag\\
    = & \, C_n \int_\Omega \dfrac{|Du(y)| \, \rho^\frac{1}{q}}{|x-y|^{n-1-\frac{\gamma}{qt}}} \dfrac{1}{\rho^\frac{1}{q}\, |x-y|^\frac{\gamma}{qt}} \, dy. \notag
\end{align}
Since $qt>1$ by assumption $(ii)$, we can apply H\"older's inequality, thus getting
\begin{align}
    |u(x)|\le & \, C_n \left(\int_\Omega  \dfrac{1}{\rho^t\, |x-y|^\gamma} \, dy \right)^\frac{1}{qt} \left( \int_\Omega \dfrac{\left(|Du(y)| \, \rho^\frac{1}{q} \right)^{(qt)'}}{|x-y|^{\left(n-1-\frac{\gamma}{qt}\right){(qt)'}}}  \,dy\right)^\frac{1}{(qt)'}, \label{stimau}
\end{align}
where $(qt)'$ denotes the conjugate exponent of $qt$.\\
By assumptions $(ii)$ and $(iii)$, it holds $n-1-\frac{\gamma}{qt}> 0 $. Then, we define $\alpha$ so that $$n-\alpha=\left(n-1-\frac{\gamma}{qt}\right)(qt)',$$ and, by virtue of assumption $(i)$, we derive that $\alpha >0$.
Thus, estimate \eqref{stimau} becomes
\begin{equation}
    |u(x)| \le C_n K^\frac{1}{qt} \left|U_\alpha \left[\left(|Du(y)| \, \rho^\frac{1}{q} \right)^{(qt)'} \right] \right|^\frac{1}{(qt)'}. \label{upunt}
\end{equation}
Since $|Du(y)| \, \rho^\frac{1}{q} \in L^q(\Omega)$, then $(|Du(y)| \, \rho^\frac{1}{q} )^{(qt)'} \in L^\frac{q}{(qt)'}(\Omega)$, where $\frac{q}{(qt)'}>1$ by assumption $(ii)$.
Now, we have that $\frac{q}{(qt)'}>\frac{n}{\alpha}$, since assumption $(iii)$ is in force. Hence, denoting with $r >1$ the exponent defined by the relation
\begin{equation}
    \dfrac{1}{r}= \dfrac{(qt)'}{q}-\dfrac{\alpha}{n}, \label{expr}
\end{equation}
by Proposition \ref{rieszthm}, we have that $$U_\alpha \left[\left(|Du(y)| \, \rho^\frac{1}{q} \right)^{(qt)'} \right]  \in L^r(\Omega)$$
and for every $m \ge (qt)'$, by estimate \eqref{upunt}, we have
\begin{equation*}
    \Vert u \Vert_{L^m(\Omega)} \le C_n K^\frac{1}{qt} \Vert U_\alpha [(|Du(y)| \, \rho^\frac{1}{q} )^{(qt)'} ] \Vert_{L^\frac{m}{(qt)'}(\Omega)}^\frac{1}{(qt)'}. 
\end{equation*}
Recalling the definition of $q^*$ at \eqref{q*} and of $r$ at \eqref{expr}, we have that $r(qt)'=q^*$. Choosing $m=r (qt)'=q^*$, by Proposition \ref{rieszthm}, we get
\begin{align*}
    \Vert u \Vert_{L^{q^*}(\Omega)} \le & \,C_n K^\frac{1}{qt} \Vert U_\alpha [(|Du(y)| \, \rho^\frac{1}{q} )^{(qt)'} ] \Vert_{L^r(\Omega)}^\frac{1}{(qt)'} \notag\\
    \le & \, C_n K^\frac{1}{qt} \Vert (|Du(y)| \, \rho^\frac{1}{q} )^{(qt)'} \Vert_{L^\frac{q}{(qt)'}(\Omega)}^\frac{1}{(qt)'} = \, C_n K^\frac{1}{qt} \Vert Du \Vert_{L^q(\Omega,\rho)}. 
\end{align*}
This concludes the proof.
\endproof

\begin{rmk}
    We notice that, by condition $(i)$ in Theorem \ref{WS}, the exponent $q^*$ defined at \eqref{q*} is such that $q^* \ge q$.
\end{rmk}

\vskip0.5cm

\noindent {{\bf Acknowledgements.} 
The authors are members of the Gruppo Nazionale per l’Analisi Matematica,
la Probabilità e le loro Applicazioni (GNAMPA) of the Istituto Nazionale di Alta Matematica (INdAM). The authors have  been supported through the INdAM - GNAMPA 2025 Project "Regolarità di soluzioni di equazioni paraboliche a crescita nonstandard degeneri" (CUP: E5324001950001).
In addition S. Russo has also been supported through the project: Sustainable Mobility Center (Centro Nazionale per la Mobilità Sostenibile – CNMS) - SPOKE 10.

\end{document}